\documentclass[12pt]{elsarticle} 
\usepackage{graphicx}
\usepackage{amsmath}
\usepackage{bm}
\usepackage{amssymb}
\usepackage{amsfonts}
\usepackage{amsthm}
\usepackage{placeins}
\usepackage{afterpage}
\newtheorem{Theorem}{Theorem}[section]
\newtheorem{Lemma}{Lemma}[section]

\newtheorem{Definition}{Definition}[section]
\newtheorem{Proposition}{Proposition}[section]
\def\ee{\end{eqnarray*}}
\def\be{\begin{eqnarray*}}
\def\bee{\end{eqnarray}}
\def\bbe{\begin{eqnarray}}
\def\ea{\end{align*}}
\def\ba{\begin{align*}}
\makeatletter
\let\today\relax
\def\ps@pprintTitle{%
    \let\@oddhead\@empty
    \let\@evenhead\@empty
    \def\@oddfoot{\footnotesize\itshape
      \hfill\today}
    \let\@evenfoot\@oddfoot
    }
\makeatother

\makeatletter
\newcommand{\bal}{\@ifstar{\@bals}{\@bal}}
\def\@bals#1\eal{\begin{align*}#1\end{align*}}
\def\@bal#1\eal{\begin{align}#1\end{align}}
\makeatletter
\def\A{\mathcal{A}} 

\def\E{E} 

\def\L{ { {\mathcal A}^\kappa } }

    \def\p{\partial}

\DeclareMathOperator*{\esssup}{ess\,sup}

\makeatletter
\newcommand{\subjclass}[2][1991]{%
  \let\@oldtitle\@title%
  \gdef\@title{\@oldtitle\footnotetext{#1 \emph{Mathematics subject classification.} #2}}%
}
\makeatother

\begin{document}

\title{$L^2$ decay estimates of weak solutions to 3D fractional MHD equations in  exterior domains  }

\author{Zhi-Min Chen}
 \ead{zmchen@szu.edu.cn}
\author{Bo-Qing Dong}\ead{bqdong@szu.edu.cn}
\author{Qiuyue Zhang}\ead{qyzhang0722@163.com}
\address{School  of Mathematical Sciences, Shenzhen University, Shenzhen 518060,  China}%


\begin{abstract} Consider three-dimensional  fractional MHD equations  in an exterior domain with the Dirichlet boundary condition  assumed. Asymptotic behaviours of weak solutions to the three-dimensional  exterior fractional MHD equations  are studied.  $L^2$ decay estimates of the weak solutions are obtained.

\end{abstract}

\maketitle

\section{Introduction}

The traditional form of incompressible magnetohydrodynamic (MHD) equations is expressed as (cf. \cite{LL, T_MHD})
\bbe
\p_t u -\Delta u +\nabla (p+\frac12 |B|^2) = -u\cdot \nabla u + B\cdot \nabla B, \label{mag1}
\\
\p_t B-\Delta B =B\cdot \nabla u- u\cdot \nabla B, \label{mag} 
\\
\nabla\cdot u=0,\,\,\, \nabla \cdot B=0,\label{mag3}
\bee
for unknown velocity field $u=(u_1,u_2,u_3)$, magnetic field $B=(B_1,B_2,B_3)$ and pressure $p$.

Mathematical investigations of the MHD equations are mainly assuming  the fluid motion in  whole spaces so that Fourier transform can be adopted. If the fluid domain $\Omega$ involves  a boundary $\p\Omega\ne \emptyset$, mathematical difficulties arise in the understanding of the magnetic equation (\ref{mag}) (cf. \cite{T_MHD, Duan}), since a suitable boundary condition with respect to the  divergence-free condition  $\nabla \cdot B=0$  and   the integral involving the  Laplacian  $-\Delta B = \nabla \times \nabla\times  B$  is
\bbe (\nabla \times B)\times n=0 \mbox{ and } \,\, B\cdot n=0 \mbox{ on } \p\Omega,\bee
 where $n$ the normal vector field of $\partial \Omega$.

In the present study, the smooth fluid domain $\Omega$   is exterior to the bounded domain $R^3\setminus \Omega$. For simplicity of analysis,  the Dirichlet boundary condition
\bbe u=0,\,\, B=0 \,\mbox{ on }\, \p \Omega, \bee
is adopted and thus  the Stokes operator
\be \A= -P\Delta \ee
 with the domain
\be  D(\A)= \left\{ w\in W^{2,2}(\Omega)^3;\,\,\nabla\cdot w=0,\,\,  w|_{\partial \Omega}=0\right\},\ee
is assumed. Here   $P$ is the Helmholtz projection operator mapping $L^2(\Omega)^3$ onto the space
\be L^2_\sigma (\Omega)^3 =\mbox{ the completion of $C^\infty_{0,\sigma}(\Omega)^3$ in $L^2(\Omega)^3$},
\ee
where $C^\infty_{0,\sigma}(\Omega)^3$ denotes the set of all divergence-free and compactly supported   smooth vector fields on $\Omega$.

The corresponding exterior  MHD problem with respect to  the Dirichlet boundary condition  was studied by Liu and Han \cite{Han}. However there is little literature on the exterior MHD problem involving fractional Stokes operators in exterior domains.
Therefore, it is the purpose of the present study to provide a primary  understanding  of the fractional  MHD equations
\begin{eqnarray}
\partial _t  u  +\A^\alpha  u &=&-P( u \cdot \nabla  u ) +P( B \cdot \nabla  B ),   \label{eq11}
\\ \partial _t  B  +\A^\beta  B &=&  - P(u\cdot \nabla B)+P(B\cdot \nabla u) \label{eq1} 
\end{eqnarray}
in the exterior domain $\Omega$,
for $0 <\alpha,\, \beta \le 1$. The fractional operators are denoted  by the spectral presentation
\bbe \A^\kappa   =\int^\infty_0 \lambda^\kappa  d E_\lambda, \bee
with the domain
\bbe  D(\A^\kappa)=\mbox{ the completion of $D(\A)$ under the norm $\|\A^\kappa u\|_{L^2} +\|u\|_{L^2}$}.\bee
Here  $E_\lambda$ represents the spectral resolution of the unit determined by the  operator $\A$ (see Yosida \cite[page 313, Theorem 1]{Yosida}).

\begin{Definition}
 Let $(u_0,B_0) \in L^2_\sigma(\Omega)^3\times L^2_\sigma(\Omega)^3$ and  $0<\alpha,\,\beta   \le 1$. A vector field  $( u , B)$  is called a weak solution of (\ref{eq11})-(\ref{eq1}) associated with the initial condition
 \bbe  u|_{t=0}=u_0,\,\,\, B|_{t=0}=B_0,\bee
  if, for given $T>0$, $(u,B) $ satisfies the following the conditions
\begin{eqnarray*} ( u ,B ) \in L^\infty\left(0,T; L^2_\sigma(\Omega)^3\times L^2_\sigma(\Omega)^3\right)\cap L^2(0,T; D(\A^{\frac\alpha 2})\times D(\A^{\frac\beta 2})),\end{eqnarray*}
\begin{eqnarray}
\lefteqn{ \int^T_0 (-\langle   u (t),\partial _t\phi(t)\rangle + \langle \A^{\frac\alpha 2} u (t),\A^{\frac\alpha 2}\phi(t)\rangle )dt -\langle u (0),\phi(0)\rangle} \nonumber
\\&=& \int^T_0 \langle  u (t)\otimes  u (t),\nabla \phi(t)\rangle dt-\int^T_0 \langle  B (t)\otimes  B(t),\nabla \phi(t)\rangle dt, \label{w1}
\end{eqnarray}
and
\begin{eqnarray}\lefteqn{ \int^T_0(- \langle  B (t),\partial _t \psi (t)\rangle + \langle \A^{\frac \beta2} B (t), \A^{\frac\beta2} \psi (t)\rangle  )dt-\langle  B (0), \psi (0)\rangle } \nonumber
\\
&=& \int^T_0 \langle  u (t)\otimes  B (t),\nabla \phi(t)\rangle dt-\int^T_0 \langle  B (t)\otimes  u(t),\nabla \phi(t)\rangle dt,\label{w2}
\end{eqnarray}
 where $\langle\cdot,\cdot \rangle$ represents the inner product of $L^2_\sigma (\Omega)^3$ and $(\phi,\psi)$ denotes   a test function  in
 \be C^1([0,T]; C_{0,\sigma}^\infty(\Omega)^3\times  C^\infty_{0,\sigma}(\Omega)^3) \mbox{ with $\phi(T)=0$ and $ \psi(T)=0$.}\ee
\end{Definition}

We are now in the position to state the main result.
\begin{Theorem}\label{Th1} Let $\Omega$ be a smooth exterior domain of $R^3$, $\frac34 < \alpha, \,\beta \le 1$ and  $(u_0,B_0) \in L^2_\sigma(\Omega)^3\times L^2_\sigma(\Omega)^3$. Then   (\ref{eq11})-(\ref{eq1}) admit  a global weak solution so that
\begin{eqnarray}\|(u(t),B(t))\|_{L^2} \to 0  \,\,\mbox{ as }\,\, t\to \infty.\label{ass1}\end{eqnarray}

Moreover, if the linear analytic semigroup has the  algebraic decay property
\begin{eqnarray}  \|(e^{-t\A^\alpha} u_0 , e^{-t\A^\beta} B_0)\|_{L^2}\le Ct^{-\gamma},\,\,\,\mbox{ for }\,\, t>0,\label{con1}\end{eqnarray}
for  $0<\gamma <\frac12$ when $\max\{ \alpha,\beta\}=1$ and  $0<\gamma \le \frac12$ when $\max\{ \alpha,\beta\}<1$,
then the nonlinear solution $(u,B)$ has the same  algebraic decay property
\begin{eqnarray}\|(u(t),B(t))\|_{L^2} \le Ct^{-\gamma},\,\,\, \mbox{ for }\,\,t>0. \label{ass2}\end{eqnarray}
\end{Theorem}



Theorem \ref{Th1} will be proved by employing the spectral representation of the exterior Stokes operator in the $L^2$ space.   The exterior Navier-Stokes flows and related problems   were studied by   Masuda \cite{1975Masuda}, Miyakawa \cite{Miya82},  Borchers and Miyakawa \cite{Miya1990,Miya1992, Miya1995}, Chen \cite{Chen1993} and Chen, Kagei and Miyakawa \cite{Chen1992}. Especially, when $B=0$ and $\alpha=1$, Theorem \ref{Th1} reduces to the $L^2$ decay result of  the exterior Navier-Stokes flows given by Miyakawa \cite{Miya1992}.
$L^2$ decay estimates of weak solutions to the
MHD equations ($\alpha=\beta=1$) in  whole spaces were  studied by Schonbek {\it et al} \cite{SSE96}
by applying the classic Fourier splitting method from Schonbek  \cite{Sc85}.

The proof of Theorem \ref{Th1} is shown in  the following two sections. Section 2 is a collection of preliminary lemmas for the proof. Section 3  consists of three subsections: existence of weak solutions, proof of assertion (\ref{ass1}) and proof of assertion (\ref{ass2}).

\section{Preliminaries}
Let $C$ and $C_n$ be generic constants, which are  independent of the quantities $ u $, $ B $, $\phi$, $\psi$,  $t>0$ and $x\in \Omega$, but they may depend on initial vector field $(u_0,B_0)$.  However, $C_n$  depends on the integer $n\ge 1$, while $C$ is independent of $n\ge 1$.

Following Sobolev imbedding is fundamental important in our analysis.
\begin{Lemma}\label{Miya} (Miyakawa \cite[Theorem 2.4]{Miya82})
Let $0\le \kappa \le 1$  and $H^\frac\kappa2$  denote the completion of $D(\A)$ with respect to the norm
$\|\A^\frac\kappa2 w\|_{L^2}$.
 Then the Sobolev imbedding
\bbe \|w\|_{L^r} \le 2\|\A^\frac\kappa2w\|_{L^2},\,\, \mbox{ for } w\in H^\frac\kappa2, \bee
holds true with respect to \be\frac1r = \frac{1-\kappa}2 + \frac\kappa 6.\ee
\end{Lemma}





\begin{Lemma}\label{L12}If $0< \kappa \le 1$ and $w \in L^2_\sigma(\Omega)^3$,  then we have
\begin{eqnarray}\|\A^\kappa   e^{-t\A^\kappa } w\|_{L^2} \le t^{-1} \|w\|_{L^2};\,\,\,\,\|  e^{-t\A^\kappa } w\|_{L^2} \le  \|w\|_{L^2},\label{X1}\end{eqnarray}
\begin{eqnarray}\label{XXX1} \|\nabla  e^{-t\A^\kappa } w\|_{L^2}\le C t^{-\frac1{2\kappa}}\|w\|_{L^2}.
\end{eqnarray}
\begin{eqnarray} \lim_{t\to \infty}\|e^{-t\A^\kappa } w\|_{L^2}=0,\label{X2}\end{eqnarray}

\end{Lemma}

This lemma is a simple consequence of the  spectral representation ( see, for example, Yosida \cite[Page 313, Theorem 1]{Yosida})
and the property of $\L$ being  positive and a self-adjoint operator.

 Equation  (\ref{X1}) is obtained  as
\begin{eqnarray} \|\L  ^ne^{-t\L } w\|_{L^2}^2= \int^\infty_0 \lambda^{2n\kappa} e^{-2t\lambda^\kappa} d\|\E_\lambda w\|_{L^2}^2\le t^{-2n} \|w\|_{L^2}^2,\,\,\, n=0, 1, \label{semi}
\end{eqnarray}
which also ensures the analytic  property of the semigroup $e^{-t\A^\kappa}$.

Similar to (\ref{semi}), we have
\begin{eqnarray} \| \nabla e^{-t\L }w\|_{L^2}=\| \A^{\frac12}e^{-t\L }w\|_{L^2}
&\le& C t^{-\frac1{2\kappa}} \|w\|_{L^2}.\label{kappa}
\end{eqnarray}
and thus obtain (\ref{XXX1}).

Decay property (\ref{X2})  readily follows from  (\ref{X1}), (\ref{XXX1}) and the density  of $C^\infty_{\sigma}(\Omega)^3$ in $L^2_\sigma(\Omega)^3$.



\section{Proof of Theorem \ref{Th1}}

\subsection{Existence of weak solutions}

With the use of the mollification operator
\begin{eqnarray}  J_n= n(n+\A)^{-1},
\end{eqnarray}
for  positive integers $n$,  we  construct approximation solutions to (\ref{eq11})-(\ref{eq1}) by solving the equations
\begin{eqnarray} \partial _t u_n + \A^\alpha u_n =  P(J_nB_n\cdot \nabla B_n-J_nu_n\cdot \nabla u_n),   \,\,\,\, u_n(0) =J_n u_0,\label{xx2}
\\
 \partial _t B_n + \A^\beta B_n =  P(J_nB_n\cdot \nabla u_n-J_n u_n\cdot \nabla B_n),   \,\,\,\, B_n(0) =J_n B_0.\label{xx22}
\end{eqnarray}

By the spectral representation and for $w\in L^2_\sigma(\Omega)^3$, we have
\begin{eqnarray} \|J_n  w \|_{L^2}\le \| w \|_{L^2}, \label{JkIk}\end{eqnarray}
\begin{eqnarray} \lim_{n\to \infty } \|J_n  w-w \|_{L^2}=0,\end{eqnarray}
and, for a constant $C_n$ dependent of $n$ and $\frac34 <\kappa \le 1$,
\begin{eqnarray} \| J_n w \|_{L^\infty}\le C_n \| w \|_{L^2}, \,\, \| \A^\kappa J_n w\|_{L^2}\le C_n \|w\|_{L^2},\,\, \label{JkIk1}
\end{eqnarray}
after the use of Sobolev imbedding.

Now we show the existence of weak solutions approached by  approximate solutions  in the  space
\begin{eqnarray*} L^\infty(0,T; W_{0,\sigma}^{1,2}(\Omega)^3\times W^{1,2}_{0,\sigma}(\Omega)^3),
\end{eqnarray*}
 where $ W_{0,\sigma}^{1,2}(\Omega)^3$ is the closure of $C^\infty_{0,\sigma}(\Omega)^3$ in the Sobolev space $W^{1,2}(\Omega)^3$.

\begin{Theorem}\label{TT1}
For $\frac34<\alpha, \,\,\beta  \le 1$ and $(u_0,B_0)\in L^2_\sigma(\Omega)^3\times L^2_\sigma(\Omega)^3$, then (\ref{eq11})-(\ref{eq1}) admit a  weak solution.
\end{Theorem}

Let us begin with the  unique existence of the approximate solutions. 
\begin{Proposition}\label{pro2}
 For $\frac34<\alpha, \,\,\beta  \le 1$,  $(u_0,B_0)\in L^2_\sigma(\Omega)^3\times L^2_\sigma(\Omega)^3$,  $T>0$ and an integer $n\ge 1$, then there exits a unique  vector field $(u_n(t),B_n(t))$ on $[0,T]$ so that
 \be (u_n,B_n) \in L^\infty(0,T; W_{0,\sigma}^{1,2}(\Omega)^3\times W_{0,\sigma}^{1,2}(\Omega)^3)
 \ee
  and solves  the integral equations
\begin{eqnarray}\label{int}
u_n(t) \!= \!e^{-t \A^\alpha} \! J_n u_0\!+\!\!\int^t_0 \!\! e^{-(t-s) \A^\alpha}\!  P(J_nB_n(s)\!\cdot\! \nabla B_n(s)\!-\!J_nu_n(s)\!\cdot\! \nabla u_n(s))ds,
\\
\label{int2}
B_n(t) \!=\! e^{-t \A^\beta} \! J_n B_0\!+\!\!\int^t_0 \!\!\! e^{-(t-s) \A^\beta}  \! P(J_nB_n(s)\!\cdot \!\nabla u_n(s)\!-\!J_nu_n(s)\!\cdot\! \nabla B_n(s))ds.
\end{eqnarray}
\end{Proposition}
\begin{proof}
To apply the contraction mapping principle with respect to a fixed integer $n\ge 1$ and a bound $M>0$, we define the complete metric space
\begin{align*} X_T = &\left\{ (u,B) \in L^\infty(0,T; W_{0,\sigma}^{1,2}(\Omega)^3\times W_{0,\sigma}^{1,2}(\Omega)^3); \,\,\right.
\\ &\left. (u(0), B(0)) =(J_nu_0, J_n B_0),\,\, \|(u,B)\|_{X_T} \le M\right\}
\end{align*}
with the norm
\begin{eqnarray*}\|(u,B)\|_{X_T} = \esssup_{0\le t\le T} (\|(u(t),B(t))\|_{L^2} + \|\nabla  (u(t),B(t))\|_{L^2}),\end{eqnarray*}
and the operators
\begin{eqnarray*}\Phi_n(u,B)(t) = e^{-t \A^\alpha} J_n u_0+\int^t_0 e^{-(t-s) \A^\alpha}  P(J_nB\cdot \nabla B-J_nu\cdot \nabla u)ds,
\\
\Psi_n(u,B)(t)= e^{-t \A^\beta} J_n B_0+\int^t_0 e^{-(t-s) \A^\beta}  P(J_nB\cdot \nabla u-J_nu\cdot \nabla B)ds.
\end{eqnarray*}


Employing Lemma \ref{L12}  and  (\ref{JkIk})-(\ref{JkIk1}), we obtain that
\begin{align} \nonumber\| e^{-t \A^\alpha} J_n u_0\|_{L^2} &+\|\nabla e^{-t \A^\alpha} J_n u_0\|_{L^2} +\| e^{-t \A^\beta} J_n B_0\|_{L^2} +\|\nabla e^{-t \A^\beta} J_n B_0\|_{L^2}\\
&\le \|u_0\|_{L^2}+\|B_0\|_{L^2}+ \|  e^{-t \A^\alpha} \A^\frac12J_n u_0\|_{L^2}+ \|  e^{-t \A^\beta} \A^\frac12J_n B_0\|_{L^2}\nonumber
\\
 &\le  C_n(\|u_0\|_{L^2}+\|B_0\|_{L^2})
\le \frac M4,\label{nnn0}
\end{align}
provided that $M$ is sufficiently large. 

Therefore,  for  $(u,B)\in  {X_T}$, we use (\ref{JkIk}),  (\ref{JkIk1}) and Lemma \ref{L12}    to produce
\begin{eqnarray*}
\lefteqn{\|\Phi_n (u,B)(t)\|_{L^2} + \|\nabla \Phi_n (u,B)(t)\|_{L^2}}\\
&\le& \| e^{-t \A^\alpha } J_n u_0\|_{L^2}+\int^t_0 \|e^{-(t-s)\A^\alpha }  P(J_nB\cdot \nabla B-J_nu\cdot \nabla u)\|_{L^2} ds
\\
&&+\| \nabla e^{-t\A^\alpha} J_n u_0\|_{L^2}+\int^t_0 \|\nabla e^{-(t-s) \A^\alpha} P(J_nB\cdot \nabla B-J_nu\cdot \nabla u)\|_{L^2} ds
\\
&\le & \frac M4+C\int^t_0 (1 +(t-s)^{-\frac1{2\alpha }}) (\| J_n u \cdot \nabla u\|_{L^2} + \| J_n B \cdot \nabla B\|_{L^2}) ds
\\
&\le &\frac M4+C\int^t_0 (1 +(t-s)^{-\frac1{2\alpha }}) (\| J_n u \|_{L^\infty} \| \nabla u\|_{L^2}+\| J_n B \|_{L^\infty} \| \nabla B\|_{L^2} ) ds
\\
&\le &\frac M4+ C_n(T +T^{1-\frac1{2\alpha }}) M^2,
\end{eqnarray*}
or
\begin{eqnarray}
\|\Phi_n (u,B)(t)\|_{L^2} + \|\nabla \Phi_n (u,B)(t)\|_{L^2}\le \frac M2,\label{my10}
\bee
for $0<t\le T$ and  $T$ sufficiently small.

Similarly, for  $(u,B),\,\, (u',B')\in {X_T}$, we have the contraction property
\begin{align}
\|\Phi_n (u,B)(t)-&\Phi_n (u',B')(t)\|_{L^2}+\|\nabla\Phi_n (u,B)(t)-\nabla \Phi_n (u',B')(t)\|_{L^2}\nonumber
\\
&\le  C_n(T +T^{1-\frac1{2\alpha }}) M\|(u,B)-(u',B')\|_{X_T}\nonumber
\\
&\le \frac14 \|(u,B)-(u',B')\|_{X_T},\label{my20}
\end{align}
when $T$ is sufficiently small.

Arguing in the same manner, we have (\ref{my10})-(\ref{my20}) with $\Phi_n$ replaced by $\Psi_n$.
Thus $(\Phi_n,\Psi_n)$ is a contraction operator mapping $ {X_T}$ into itself, and so $(\Phi_n, \Psi_n)$ admits a unique fixed point $(u_n,B_n)\in  {X_T}$ satisfying  integral equations (\ref{int})-(\ref{int2}). We thus obtain local existence result.
This integral equation solution also solves the differential equations  (\ref{xx2}) and (\ref{xx22}). This is   confirmed by   the $L^2$ theory of linear parabolic equation given by Simon \cite{Simon}, which also   ensures that
\begin{eqnarray*}
(u_n,B_n) \in L^2\left(0, T; D(\A^\alpha )\times D(\A^\beta)\right),
\\
 (\partial _t u_n, \p_t B_n) \in
L^2\left(0, T; L^2_\sigma(\Omega)^3 \times L^2_\sigma(\Omega)^3\right).
\end{eqnarray*}

 It is readily seem that the local solution can be extended globally, if
 \bbe \|(u_n,B_n)\|_{X_T}<\infty,\label{XTXT} \bee
  whenever $(u_n,B_n)$ solves (\ref{xx2}) and (\ref{xx22}) on the open time interval $[0,T)$ for $T>0$.
 Indeed, taking the inner product of  (\ref{xx2})-(\ref{xx22}) with $(u_n,B_n)$ and integrating parts by the divergence free condition of $J_nu$ and $J_nB$, we have
\begin{eqnarray} \langle \partial _t u_n, u_n\rangle  + \langle \partial _t  B_n,B_n\rangle  +\langle \A^\alpha  u_n, u_n\rangle  + \langle \A^\beta  B_n, B_n\rangle =0.  \label{xx2x}
\end{eqnarray}
Hence,  we have the energy inequality,  after the integration with respect to $t$ and the use of (\ref{JkIk}),
\begin{align}
\|u_n(t)\|_{L^2}^2+&\|B_n(t)\|_{L^2}^2 +2\int^t_0(\|\A^\frac\alpha 2 u _n\|_{L^2}^2+\|\A^\frac\beta 2 B_n\|_{L^2}^2)ds\nonumber
\\
&\le  \|J_nu_0\|_{L^2}^2+\|J_nB_0\|_{L^2}^2\le \|u_0\|_{L^2}^2+ \|B_0\|_{L^2}^2.\label{new112}
\end{align}
This shows the uniform boundedness of $\|(u_n,B_n)(t)\|_{L^2}$.

To obtain  the boundedness of $\|\nabla (u_n,B_n)(t)\|_{L^2}$,  we apply    (\ref{XXX1}),   (\ref{JkIk1})  and (\ref{new112})  into (\ref{int})-(\ref{int2}) to produce
\begin{align*}
\|\nabla (u_n,B_n)(t)\|_{L^2}
&\le  C_n\|(u_0,B_0)\|_{L^2}
\\
&\!+\!C_n\!\!\int^t_0\!\! \!((t\!-\!s)^{-\frac1{2\alpha }}\!+\!(t\!-\!s)^{-\frac1{2\beta }}\!) \| (u_n,B_n) \|_{L^2} \|\nabla (u_n,B_n)\|_{L^2} ds
\\
&\le  C_n+C_n\int^t_0 ((t-s)^{-\frac1{2\alpha}}+(t-s)^{-\frac1{2\beta }})  \|\nabla (u_n,B_n)\|_{L^2} ds.
\end{align*}
Hence, using the Gronwall inequality, we have the desired bound
\begin{eqnarray} \|\nabla (u_n,B_n)(t)\|_{L^2} &\le & C_n \exp(C_n T^{1-\frac1{2\alpha }}+C_n T^{1-\frac1{2\beta }}),\,\, 0<t<T.\label{nana}\end{eqnarray}
Therefore, the combination of (\ref{new112}) and (\ref{nana})  gives (\ref{XTXT}). Thus the local solution $(u_n,B_n)$ can be extended to the time interval $[0,T]$ for any $T>0$.
The proof of Proposition \ref{pro2} is complete.
\end{proof}

\noindent{\bf Proof of Theorem \ref{TT1}}.
To obtain the  weak solution existence, we need to study the compactness of the sequence $(u_n,B_n)$. For
$ \phi\in W^{1,2}_{0,\sigma}(\Omega)^3,$
 taking the inner product of (\ref{int}) with $\phi$ and integrating by parts, we have
\begin{align}
\langle \partial _tu_n, \phi\rangle
 &=-\langle \A^\alpha u_n, \phi\rangle  +\langle J_nB_n\cdot \nabla B_n-J_nu_n\cdot \nabla u_n,\phi\rangle\nonumber
\\
&= -\langle  \A^{\alpha -\frac12} u_n, \A^\frac12 \phi\rangle
- \langle J_n u _n\otimes  u _n,\nabla \phi\rangle+\langle J_n B _n\otimes  B _n,\nabla \phi\rangle\nonumber
\\
&\le (\| \A^{\alpha -\frac12} u _n\|_{L^2}+
\|J_n u_n\|_4\|u_n\|_4+\|J_n B_n\|_4\|B_n\|_4
)\|\nabla \phi\|_{L^2}.\label{my1}
\end{align}
It follows from H\"older inequality and   Lemma \ref{Miya}  that
\bbe \|w\|_{L^4} \le \|w\|_{L^2}^{1-\frac3{4\kappa}}\| w\|_{L^{1/(\frac12-\frac\kappa 3)}}^\frac3{4\kappa}
\le C\|w\|_{L^2}^{1-\frac3{4\kappa}}\| \A^\frac\kappa2 w\|_{L^2}^\frac3{4\kappa}, \mbox{ for } \kappa >\frac34.\label{n39}
\bee
Therefore, by (\ref{JkIk}), (\ref{new112}) and an interpolation inequality, we see that  (\ref{my1}) becomes
\begin{eqnarray}
\|\partial _t u_n\|_{W^{-1,2}}\nonumber
&\le& C\|  u_n\|_{L^2}^{\frac1\alpha -1}\|\A^{\frac\alpha 2} u_n\|_{L^2}^{2-\frac1\alpha }
\\
&&+\nonumber
C\| u_n\|_{L^2}^{2-\frac3{2\alpha }}\|\A^\frac\alpha 2 u_n\|_{L^2}^{\frac3{2\alpha }}+
C\| B_n\|_{L^2}^{2-\frac3{2\beta }}\|\A^\frac\beta 2 B_n\|_{L^2}^{\frac3{2\beta }}
\\
&\le& C(\|\A^{\frac\alpha 2} u_n\|_{L^2}^{2-\frac1\alpha }+\|\A^{\frac\alpha 2} u_n\|_{L^2}^{\frac3{2\alpha }}+ \|\A^{\frac\beta 2} B_n\|_{L^2}^{\frac3{2\beta }}).\label{my2}
\end{eqnarray}
Arguing in the same manner  as in the derivation of  (\ref{my2}), we have
\begin{eqnarray} \|\partial _t B_n\|_{W^{-1,2}}\label{my3}
&\le&  C(\|\A^{\frac\beta 2} B_n\|_{L^2}^{2-\frac1\beta }+\|\A^{\frac\alpha 2} u_n\|_{L^2}^{\frac3{2\alpha }}+\|\A^{\frac\beta 2} B_n\|_{L^2}^{\frac3{2\beta }}).\end{eqnarray}
Thus, the combination of  (\ref{new112}), (\ref{my2}) and (\ref{my3}) yields the boundedness of
\begin{eqnarray} \partial _t u_n, \p_t B_n \in L^{\kappa}(0,T; W^{-1,2}(\Omega)^3) \mbox{ for } \kappa =\min\{\frac2{2-\frac1\alpha},\,\frac2{2-\frac1\beta}, \, \frac{4\alpha}3,\,\frac{4\beta}3\}
\end{eqnarray}
uniformly with respect to  all $n\ge 1$
 for any fixed  $T>0$. It follows from (\ref{new112}) that the sequence  $(u_n,B_n)$ is uniformly bounded with respect to integers $n\ge 1$  in the space
\begin{eqnarray} L^\infty(0,T; L^2_\sigma(\Omega)^3\times L^2_\sigma(\Omega)^3)\cap L^2 (0,T; D(\A^{\frac\alpha 2})\times D(\A^\frac\beta2)).\label{space}
\end{eqnarray}
Therefore,
 $(u_n,B_n)$ admits a subsequence, denoted again by $(u_n,B_n)$, converging to an element   $(u,B)$ in the space given by (\ref{space}) in the following sense
\begin{eqnarray*}
&&(u_n,B_n)  \rightarrow (u,B) \,\,\mbox{weak-star  in }\,\, L^\infty(0,T; L^2_\sigma(\Omega)^3\times L^2_\sigma(\Omega)^3),
\\
&& (u_n,B_n) \rightarrow (u,B) \,\,\mbox{ weakly in }\,\, L^2 (0,T; D(\A^{\frac\alpha 2})\times D(\A^\frac\beta2)).
\end{eqnarray*}
 Furthermore, due to the compact result from Temam \cite[Theorem 2.1 in page 271]{T}, we see that 
\begin{eqnarray*}   (u_n,B_n) \rightarrow (u,B) \,\,\mbox{ strongly in }\,\, L^2 (0,T; L^2(K)^3 \times L^2(K)^3),
\end{eqnarray*}
 for any compact set $K \subset \Omega$.
This convergence implies  that   $(u,B)$ is a desired weak solution  solving (\ref{w1}) and (\ref{w2}).

The proof of Theorem \ref{TT1} is complete.

\subsection{Proof of assertion  (\ref{ass1})}

From the convergence of the sequence $(u_n,B_n)$ in  the previous subsection, we see that
\begin{eqnarray} \|(u,B)(t)\|_{L^2} \le \liminf_{n\to \infty}\|(u_n,B_n)(t)\|_{L^2},
\end{eqnarray}
for  $t>0$. Thus it  suffices to  show corresponding $L^2$ decay estimates of the approximate solutions  $(u_n,B_n)$ uniformly with respect to $n\ge 1$.  Therefore, for simplicity of notation, we omit the subscript $n$ for the approximate solutions. Thus  approximate  equations
(\ref{int}) and (\ref{int2}) are  rewritten, respectively,  as
\begin{eqnarray}\label{intt}
u(t) = e^{-t \A^\alpha} J_n u_0+\int^t_0 e^{-(t-s) \A^\alpha}  P(J_nB\cdot \nabla B-J_nu\cdot \nabla u)ds,
\\
\label{intt2}
B(t) = e^{-t \A^\beta} J_n B_0+\int^t_0 e^{-(t-s) \A^\beta}  P(J_nB\cdot \nabla u-J_nu\cdot \nabla B)ds.
\end{eqnarray}
According to the proof of (\ref{new112}),   energy inequality (\ref{new112}) can be rephrased as
\begin{eqnarray} \nonumber\|u(t)\|_{L^2}^2+\|B(t)\|_{L^2}^2 + \int^t_s(\|\A^{\frac\alpha 2} u (r)\|_{L^2}^2+\|\A^{\frac\beta 2} B (r)\|_{L^2}^2)dr
\\
\le \|u(s)\|_{L^2}^2+\|B(s)\|_{L^2}^2,\label{energy}\end{eqnarray}
for  $0\le s <t$.

Let us begin with the estimate of the nonlinear integrand of (\ref{intt}) in the $L^2$ norm.
Applying the divergence-free property   and  integrating by parts, we take a test function $\phi\in L^2_\sigma (\Omega)^3$ to obtain
 \begin{align} \langle e^{-t \A^\alpha }& P(J_nB\cdot \nabla B-J_nu\cdot \nabla u), \phi\rangle\nonumber
 \\
 &= \langle J_n u \otimes  u ,\nabla e^{-t \A^\alpha}\phi\rangle- \langle J_n B \otimes  B ,\nabla e^{-t \A^\alpha} \phi\rangle,\label{xy1}
 \end{align}
Applying  H\"older inequality to the previous identity and employing the estimates (\ref{XXX1}), (\ref{JkIk}) and (\ref{n39}),
we have
\begin{align}
\|e^{-t \A^\alpha }& P(J_nB\cdot \nabla B-J_nu\cdot \nabla u)\|_{L^2}\nonumber
\\
&\le Ct^{-\frac1{2\alpha}} (\|J_n u \|_{L^4}  \|u\|_{L^4}+\|J_n B \|_{L^4}  \|B\|_{L^4})\nonumber
\\
&\le Ct^{-\frac1{2\alpha }}(\| u \|_{L^2}^{2-\frac3{2\alpha }} \| \A^{\frac\alpha 2} u \|_{L^2}^{\frac3{2\alpha }}
+\| B \|_{L^2}^{2-\frac3{2\beta }} \| \A^{\frac\beta 2} B \|_{L^2}^{\frac3{2\beta }}). \label{mymy12}
\end{align}
With the  application of (\ref{mymy12}), the   $L^2$ estimate of (\ref{intt}) becomes 
\begin{align}
 &\|u(t)\|_{L^2}-  \|  e^{-t\A^\alpha}u_0\|_{L^2}\label{mymy2}
 \\
&\le C\int^t_0 (t-s)^{-\frac1{2\alpha }}(\| u (s)\|_{L^2}^{2-\frac3{2\alpha }} \| \A^{\frac\alpha 2} u (s)\|_{L^2}^{\frac3{2\alpha }}
+\| B (s)\|_{L^2}^{2-\frac3{2\beta }} \| \A^{\frac\beta 2} B (s)\|_{L^2}^{\frac3{2\beta }})ds.\nonumber
\end{align}
Arguing in the same manner with the derivation of (\ref{mymy2}), we have the $L^2$ estimate of (\ref{intt2}) derived as
\begin{align}
 \|B&(t)\|_{L^2}-  \|  e^{-t\A^\beta}B_0\|_{L^2}\label{mymy1}
\\
&\le C\int^t_0 (t-s)^{-\frac1{2\beta }}\| u (s)\|_{L^2}^{1-\frac3{4\alpha }} \| \A^{\frac\alpha 2} u (s)\|_{L^2}^{\frac3{4\alpha }}
\| B (s)\|_{L^2}^{1-\frac3{4\beta }} \| \A^{\frac\beta 2} B (s)\|_{L^2}^{\frac3{4\beta }}ds.\nonumber
\end{align}
Integrating (\ref{mymy2})-(\ref{mymy1}) respectively,   we have
\begin{align*}
 &\int^t_0\|u(s)\|_{L^2}ds+ \int^t_0\|B(s)\|_{L^2}ds-\int^t_0\|  e^{-s \A^\alpha}u_0\|_{L^2}ds-\int^t_0\|  e^{-s \A^\beta}B_0\|_{L^2}ds
 \\
 &\le  \!\nonumber C\!\!\!\int^t_0\!\int^r_0\!\! (r\!-\!s)^{-\frac1{2\alpha }}( \| u (s)\|_{L^2}^{2-\frac3{2\alpha }} \| \A^{\frac\alpha 2} u (s)\|_{L^2}^{\frac3{2\alpha }}\!+\!\| B (s)\|_{L^2}^{2-\frac3{2\beta }} \| \A^{\frac\beta 2} B (s)\|_{L^2}^{\frac3{2\beta }})dsdr\nonumber
\\
&+C\int^t_0 \int^r_0(r-s)^{-\frac1{2\beta }}\| u (s)\|_{L^2}^{1-\frac3{4\alpha }} \| \A^{\frac\alpha 2} u (s)\|_{L^2}^{\frac3{4\alpha }}
\| B (s)\|_{L^2}^{1-\frac3{4\beta }} \| \A^{\frac\beta 2} B (s)\|_{L^2}^{\frac3{4\beta }}dsdr.\nonumber
\end{align*}
After changing the integration order and integrating with respect to $r$, we have
\begin{align}
&\int^t_0\|u(s)\|_{L^2}ds+ \int^t_0\|B(s)\|_{L^2}ds-\int^t_0\|  e^{-s \A^\alpha}u_0\|_{L^2}ds-\int^t_0\|  e^{-s \A^\beta}B_0\|_{L^2}ds
\nonumber
\\
&\le   Ct^{1-\frac1{2\alpha }}\!\int^t_0( \| u(s)\|_{L^2}^{2-\frac3{2\alpha }} \| \A^{\frac\alpha 2} u(s)\|_{L^2}^{\frac3{2\alpha }}+\| B(s)\|_{L^2}^{2-\frac3{2\beta }} \| \A^{\frac\beta 2} B(s)\|_{L^2}^{\frac3{2\beta}})ds\!\nonumber
\\
&+Ct^{1-\frac1{2\beta }}\int^t_0 \| u (s)\|_{L^2}^{1-\frac3{4\alpha }} \| \A^{\frac\alpha 2} u (s)\|_{L^2}^{\frac3{4\alpha }}
\| B (s)\|_{L^2}^{1-\frac3{4\beta }} \| \A^{\frac\beta 2} B (s)\|_{L^2}^{\frac3{4\beta }}ds.
\label{decay2}
\end{align}
By  using H\"older inequality  in (\ref{decay2}) and then taking  energy inequality (\ref{energy}) into account, the $L^2$ estimate becomes
\begin{align}
&\int^t_0\|u(s)\|_{L^2}ds+ \int^t_0\|B(s)\|_{L^2}ds-\int^t_0\|  e^{-s \A^\alpha}u_0\|_{L^2}ds-\int^t_0\|  e^{-s \A^\beta}B_0\|_{L^2}ds
\nonumber
\\
 & \leq   Ct^{1-\frac1{2\alpha }}t^{1-\frac3{4\alpha}}\left(\int^t_0\| u (s)\|_{L^2}^{\frac{8\alpha }3-2}\| \A^{\frac\alpha 2} u(s)\|_{L^2}^{2} ds\right)^{\frac3{4\alpha }}\nonumber
 \\
 & + Ct^{1-\frac1{2\alpha }}t^{1-\frac3{4\beta}}\left(\int^t_0\| B (s)\|_{L^2}^{\frac{8\beta }3-2}\| \A^{\frac\beta 2} B(s)\|_{L^2}^{2} ds\right)^{\frac3{4\beta }}\nonumber
 \\
 & + Ct^{1-\frac1{2\beta }}t^{1-\frac3{8\alpha}-\frac3{8\beta}}\left(\int^t_0\| u \|_{L^2}^{\frac{8\alpha }3-2}\| \A^{\frac\alpha 2} u\|_{L^2}^{2} ds\right)^{\frac3{8\alpha }}\!\!\!\left(\int^t_0\| B \|_{L^2}^{\frac{8\beta }3-2}\| \A^{\frac\beta 2} B\|_{L^2}^{2} ds\right)^{\frac3{8\beta }}\nonumber
 \\
  & \le C[t^{2-\frac5{4\alpha }}+t^{2-\frac1{2\alpha }-\frac3{4\beta}}+t^{2-\frac3{8\alpha}-\frac7{8\beta}}]\|(u_0,B_0)\|_{L^2}^2.\label{decayy}
\end{align}

On the other hand, it follows from (\ref{energy}) that
\begin{eqnarray} \|(u,B)(t)\|_{L^2} \le \frac1t\int^t_0 \|(u,B)(s)\|_{L^2}ds. \label{decay5}\end{eqnarray}
Therefore,  equation (\ref{decayy}) becomes
\begin{align}
 \|(u,B)(t)\|_{L^2} & \le \frac1t\int^t_0\|  e^{-s\A^\alpha}u_0\|_{L^2}ds+\frac1t\int^t_0\|  e^{-s \A^\beta}B_0\|_{L^2}ds\nonumber
 \\
 &+ C(t^{1-\frac5{4\alpha }}+t^{1-\frac1{2\alpha }-\frac3{4\beta}}+t^{1-\frac3{8\alpha}-\frac7{8\beta}})
 \nonumber
 \\
 &\le  \frac1t\int^t_0\|  e^{-s\A^\alpha}u_0\|_{L^2}ds+\frac1t\int^t_0\|  e^{-s \A^\beta}B_0\|_{L^2}ds+Ct^{-\frac1{4\alpha }},\label{decay1}
\end{align}
since  $t>1$  is assumed  due to the energy inequality  (\ref{energy}) and
\begin{eqnarray}\label{o1}
\frac5{4\alpha }-1\ge \frac1{4\alpha }, \,\,\, \frac1{2\alpha }+\frac3{4\beta }-1\ge \frac1{4\alpha }, \,\,\,  \frac3{8\alpha }+\frac7{4\beta }-1\ge \frac1{4\alpha },
\end{eqnarray}
due to  $\alpha, \beta \le 1$.
This shows  the desired  $L^2$ decay (\ref{ass1}) due to Lemma \ref{L12}.

\subsection{Proof of assertion (\ref{ass2})}

By assumption (\ref{con1}),
 equation (\ref{decay1}) implies
\begin{eqnarray} \|(u,B)(t)\|_{L^2} \le C t^{-\gamma} + C t^{-\frac1{4\alpha }}.
\end{eqnarray}
Here, we only consider the case $t\ge 1$,  due to the uniform bound  \begin{eqnarray*}\|(u,B)(t)\|_{L^2}\le \|(u_0,B_0)\|_{L^2}\end{eqnarray*} given by  (\ref{energy}). Hence, we obtain (\ref{ass2}) when $\gamma \le \frac1{4\alpha }$.

For $\gamma >\frac1{4\alpha }$,
since $1\ge \alpha  >\frac34$, we set $p, q, r>1$  so that
\begin{eqnarray} \label{dd}\frac1p+ \frac3{4\alpha }=1,\,\,\frac1q+ \frac3{4\beta }=1,\,\,\,
 \frac1r + \frac 3{8\alpha}+ \frac3{8\beta}=1.
\end{eqnarray}
By   (\ref{con1}), (\ref{decay5}), (\ref{energy}) and  H\"older inequality ,   estimate    (\ref{decay2})  becomes
\begin{align}\nonumber
&\|(u,B)(t)\|_{L^2}\nonumber
\\
 &\le Ct^{-\gamma}
 +Ct^{-\frac1{2\alpha }}\nonumber
 \left(\int^t_0\|  u\|_{L^2}^{(2-\frac3{2\alpha })p}ds\right)^{\frac1p}\left(\int^t_0\| \A^{\frac\alpha 2} u\|_{L^2}^{2}ds \right)^{\frac3{4\alpha }}
\\&+Ct^{-\frac1{2\alpha }}\nonumber
 \left(\int^t_0\!\!\|  B\|_{L^2}^{(2-\frac3{2\beta })q}ds\right)^{\frac1q}\left(\int^t_0\!\!\| \A^{\frac\beta 2} B\|_{L^2}^{2}ds \right)^{\frac3{4\beta}}
   \\ & +Ct^{-\frac1{2\beta }}
 (\int^t_0\!\!\|  u\|_{L^2}^{(1-\frac3{4\alpha })r}\|  B\|_{L^2}^{(1-\frac3{4\beta })r}ds)^{\frac1r}
 (\int^t_0\!\!\| \A^{\frac\alpha 2} u\|_{L^2}^{2}ds )^{\frac3{8\alpha}}(\int^t_0\!\!\| \A^{\frac\beta 2} B\|_{L^2}^{2}ds )^{\frac3{8\beta}}\nonumber
 \\
 &\le Ct^{-\gamma}
 +Ct^{-\frac1{2\alpha }}\nonumber
 \left(\int^t_0\|  u\|_{L^2}^{(2-\frac3{2\alpha })p}ds\right)^{\frac1p}
+Ct^{-\frac1{2\alpha }}\nonumber
 \left(\int^t_0\|  B\|_{L^2}^{(2-\frac3{2\beta })q}ds\right)^{\frac1q}
   \\ & +Ct^{-\frac1{2\beta }}
 \left(\int^t_0\|  u\|_{L^2}^{(1-\frac3{4\alpha })r}\|  B\|_{L^2}^{(1-\frac3{4\beta })r}ds\right)^{\frac1r}.\label{decay7}
 \end{align}

 The desired decay estimate will be derived by a  boot strap iteration scheme. To do so, we begin with the step   $\gamma_1=\frac1{4\alpha }$ so that
 \begin{eqnarray} \|(u,B)(t)\|_{L^2}\le C t^{-\gamma_1}.\end{eqnarray}
   Assuming the existence of $\gamma_n$ so that   $0<\gamma_n <\gamma\le \frac12$  for a positive  integer $n$  and   there holds the decay estimate
\begin{eqnarray} \|(u,B)(t)\|_{L^2} \le C t^{-\gamma_n},\label{my000}\end{eqnarray}
we will show  that the previous estimate can be upgrade to the one involving  a  power to be defined as  $-\gamma_{n+1}$. Indeed, the application of  (\ref{dd}) and  (\ref{my000}) into    (\ref{decay7}), we have
 \begin{align}
 &\|(u,B)(t)\|_{L^2}\nonumber
 \\
 &\le Ct^{-\gamma}
 \!+\! Ct^{-\frac1{2\alpha }}\nonumber
 \left(\int^t_0 s^{-\gamma_n(2-\frac3{2\alpha })p}ds\right)^{\frac1p}
\!+\! Ct^{-\frac1{2\alpha }}\nonumber
 \left(\int^t_0s^{-\gamma_n(2-\frac3{2\beta })q}ds\right)^{\frac1q}
   \\ & +Ct^{-\frac1{2\beta }}
 \left(\int^t_0 s^{-\gamma_n(1-\frac3{4\alpha })r}s^{-\gamma_n(1-\frac3{4\beta })r}ds\right)^{\frac1r}\nonumber 
 \\
 &\le C [t^{-\gamma} +t^{-\frac1{2\alpha}}(t^{\frac1p -\gamma_n(2-\frac3{2\alpha })}+t^{\frac1q -\gamma_n(2-\frac3{2\beta })})+t^{-\frac1{2\beta}}t^{\frac1r -\gamma_n(2-\frac3{4\alpha }-\frac3{4\beta})}
 \nonumber
  ]
  \\
  &= C[t^{-\gamma} \!+\!t^{1-\frac5{4\alpha } -\gamma_n(2-\frac3{2\alpha })}\!+\! t^{1-\frac2{4\alpha}-\frac3{4\beta } -\gamma_n(2-\frac3{2\beta })}
  \!+\! t^{1-\frac3{8\alpha }-\frac7{8\beta} -\gamma_n(2-\frac3{4\alpha }-\frac3{4\beta})}\label{xyz}
  ],
 \end{align}
 since $0<2\gamma_n <1$, $\alpha, \beta>\frac34$ and
 \be \gamma_n(2-\frac3{2\alpha})p=\gamma_n(2-\frac3{2\beta})q=\gamma_n(2-\frac3{4\alpha}-\frac3{4\beta})r
=2\gamma_n <1.
\ee
due to (\ref{dd}).
Rewrite (\ref{xyz}) as
\begin{eqnarray}
\|(u,B)(t)\|_{L^2}
  &\le & C(t^{-\gamma} +t^{-a_n}+t^{-b_n}
  +t^{-c_n}\label{xyzaa}
  ),
 \end{eqnarray}
with
 \begin{eqnarray*} a_{n} = (\frac5{4\alpha }-1) +\gamma_n(2-\frac3{2\alpha }),
  \\
   b_{n} = (\frac2{4\alpha }+\frac3{4\beta }-1) +\gamma_n(2-\frac3{2\beta }),
   \\
   c_{n}=\frac3{8\alpha }+\frac7{8\beta}-1 +\gamma_n(2-\frac3{4\alpha }-\frac3{4\beta}).
\end{eqnarray*}
We see that
\bbe
a_{n}-b_{n} = \frac3{4\alpha}-\frac3{4\beta}-2\gamma_n(\frac3{4\alpha}-\frac3{4\beta})=\frac34(\frac1{\alpha}-\frac1{\beta})(1-2\gamma_n),\label{m1}
\\
a_{n}-c_{n} = \frac7{8\alpha}-\frac7{8\beta}-\gamma_n(\frac3{4\alpha}-\frac3{4\beta})=\frac34(\frac1{\alpha}-\frac1{\beta})(\frac76-\gamma_n),\label{m2}
\\
b_{n}-c_{n} = \frac1{8\alpha}-\frac1{8\beta}+\gamma_n(\frac3{4\alpha}-\frac3{4\beta})\label{m3}
=\frac34 (\frac1{\alpha}-\frac1{\beta})(\frac16+\gamma_n).
\bee
This implies
\be a_{n} \ge b_{n}\ge c_{n}, \,\mbox{ when } \frac1\alpha-\frac1\beta>0.
\ee
 Hence we set
\bbe \gamma_{n+1} = c_n, \mbox{ for } \frac1\alpha-\frac1\beta>0.\bee
This together with (\ref{xyzaa}) gives
\bbe \|(u,B)(t)\|_{L^2} \le Ct^{-\gamma}+ Ct^{-\gamma_{n+1}}.\label{my4}\bee
The definition of $\gamma_{n+1}$ shows
\begin{eqnarray*} \gamma_{n+1}=( \frac3{8\alpha }+\frac7{8\beta}-1) \sum_{m=0}^{n-1}(2-\frac3{4\alpha }-\frac3{4\beta})^m + \gamma_1 (2-\frac3{4\alpha }-\frac3{4\beta})^n.
\ee
The limit of this summation is, as $n\to \infty$,
\be
\frac{ \frac3{8\alpha }+\frac7{8\beta}-1}{1-(2-\frac3{4\alpha }-\frac3{4\beta})}=\frac{ \frac3{8\alpha }+\frac7{8\beta}-1}{\frac3{4\alpha }+\frac3{4\beta}-1}\,\,\, \left\{ \begin{array}{c}=\frac12, \mbox{ when } \beta =1,\vspace{2mm}\\
>\frac12, \mbox{ when } \beta <1.\end{array}
\right.
\end{eqnarray*}
Hence, there is a positive  integer $n_0$ so that  $\gamma_{n_0+1} \ge \gamma$. Thus (\ref{my4})  becomes
\bbe \|(u,B)(t)\|_{L^2} \le Ct^{-\gamma} + Ct^{-\gamma_{n_0+1}}\le Ct^{-\gamma}.\label{my5}
\bee

On the other hand, when $\frac1\alpha-\frac1\beta\le 0$, we see from (\ref{m1})-(\ref{m3}) that
\bbe a_{n} \le b_{n}\le c_n.\label{my7}\bee
 We thus define $\gamma_{n+1} = a_n$ and so
\be \gamma_{n+1} &=& (\frac5{4\alpha }-1) +\gamma_n(2-\frac3{2\alpha })
  \\
  &=&(\frac5{4\alpha }-1) \sum_{m=0}^{n-1}(2-\frac3{2\alpha })^m+\gamma_1(2-\frac3{2\alpha })^n.
\ee
The limit of this summation is, as $n\to \infty$,
\begin{eqnarray} (\frac5{4\alpha }-1)\sum_{m=0}^\infty (2-\frac3{2\alpha })^m
=\frac{5-4\alpha}{6-4\alpha}\,\,\, \left\{ \begin{array}{c}=\frac12, \mbox{ when } \alpha =1,\vspace{2mm}\\
>\frac12, \mbox{ when } \alpha <1.\end{array}
\right.
\end{eqnarray}
This together with (\ref{xyzaa}) and (\ref{my7})  implies the existence of an integer $n_0\ge 1$ so that  $\gamma_{n_0+1} \ge \gamma$ and hence
\begin{eqnarray} \|(u,B)(t)\|_{L^2}\le C(t^{-\gamma} + t^{-\gamma_{n_0+1}} )\le C t^{-\gamma}.\end{eqnarray}
This gives (\ref{ass2}) and completes the proof of Theorem \ref{Th1}.

\

\noindent {\bf Acknowledgment}.
This research is supported by The Shenzhen Natural Science Fund of China ( Stable Support Plan Program No. 20220805175116001).

\

\end{document}